\documentclass[11pt]{article}
\usepackage{amsmath}
\usepackage{amsthm}
\usepackage{mathrsfs}

\usepackage{amssymb,amsfonts}
\usepackage{hyperref}
\usepackage{latexsym}
\usepackage{todonotes}
\usepackage{nicefrac}
\usepackage{epsfig}
\usepackage{stmaryrd}
\usepackage{setspace}
\usepackage{enumerate}
\usepackage[all]{xypic}
\usepackage{bbm,ifpdf}
\usepackage{tikz-cd}
\usepackage{tikz}
\usepackage{tikzsymbols}
\ifpdf
\usepackage{pdfsync}
\fi

\newtheorem{theorem}{Theorem}[section]

\newtheorem{lemma}[theorem]{Lemma}
\newtheorem{proposition}[theorem]{Proposition}
\newtheorem{corollary}[theorem]{Corollary}

{
\theoremstyle{definition}

\newtheorem{remark}[theorem]{Remark}
}

\newcommand{\excise}[1]{}

\renewcommand{\and}{\qquad\text{and}\qquad}

\newcommand{\bin}[2]{
 \left(
   \begin{array}{@{}c@{}}
      #1 \\ #2
  \end{array}
 \right)  }

\newcommand{\N}{\mathbb{N}}

\newcommand{\SQSP}{\mathcal{S}}

\newcommand{\Des}{\operatorname{Des}}
\newcommand{\RDes}{\operatorname{RDes}}

\allowdisplaybreaks

\begin{document}
\spacing{1.2}
\noindent{\Large\bf The combinatorics behind the leading Kazhdan--Lusztig coefficients of braid matroids}\\

\noindent{\bf Alice L.L. Gao}\footnote{Supported by the National Science Foundation of China (11801447) and by the Natural Science Foundation of Shaanxi Province (2020JQ-104).}\\
School of Mathematics and Statistics, Northwestern Polytechnical University, Xi'an, Shaanxi 710072, China\\
llgao@nwpu.edu.cn
\vspace{.1in}

\noindent{\bf Nicholas Proudfoot}\footnote{Supported by NSF grants DMS-1954050, DMS-2039316, and DMS-2053243.}\\
Department of Mathematics, University of Oregon, Eugene, OR 97403, USA\\
njp@uoregon.edu

\vspace{.1in}

\noindent{\bf Arthur L.B. Yang}\footnote{Supported in part by the Fundamental Research Funds for the Central Universities and the National Natural Science Foundation of China (11971249 and 12325111).}\\
Center for Combinatorics, LPMC, Nankai University, Tianjin 300071, China\\
yang@nankai.edu.cn
\vspace{.1in}

\noindent{\bf Zhong-Xue Zhang}\\
Center for Combinatorics, LPMC, Nankai University, Tianjin 300071, China\\
zhzhx@mail.nankai.edu.cn
\\

{\small
\begin{quote}
\noindent {\em Abstract.}
Ferroni and Larson gave a combinatorial interpretation of the braid Kazhdan--Lusztig polynomials in terms of series-parallel matroids. As a consequence, they confirmed an explicit formula for the leading Kazhdan--Lusztig coefficients of braid matroids with odd rank, as conjectured by Elias, Proudfoot, and Wakefield. Based on Ferroni and Larson's work, we further explore the combinatorics behind the leading Kazhdan--Lusztig coefficients of braid matroids. The main results of this paper include an explicit formula for the leading Kazhdan--Lusztig coefficients of braid matroids with even rank, a simple expression for the number of simple series-parallel matroids of rank $k+1$ on $2k$ elements, and explicit formulas for the leading coefficients of inverse Kazhdan--Lusztig polynomials of braid matroids. The binomial identity for the Abel polynomials plays an important role in the proofs of these formulas.
\end{quote} }

\noindent \emph{AMS Classification 2020:} {05B35, 05A15, 05A19} 


\section{Introduction}

Given a matroid $M$ of positive rank, the {\bf Kazhdan--Lusztig polynomial} $P_M(t)\in\N[t]$ is a polynomial with non-negative integer coefficients of degree strictly less than half the rank of $M$ \cite{EPW,BHMPW2}.
Though these polynomials are characterized by a simple recursive formula, computing them in practice for explicit families of matroids is often quite difficult.
Let $B_n$ be the {\bf braid matroid}, associated with the complete graph on $n$ vertices.  Ferroni and Larson \cite{FL-braid} gave a beautiful combinatorial interpretation of the coefficients of $P_{B_n}(t)$, which we now describe.

A matroid is called {\bf quasi series-parallel} if
it is isomorphic to a direct sum of loops and matroids associated with series-parallel graphs.
Equivalently, it is quasi series-parallel if
it does not contain any minor isomorphic to the braid matroid $B_4$ or the uniform matroid of rank two on four elements \cite[Proposition 2.6]{FL-braid}. For a finite set $E$, let $\SQSP(E,k)$ denote the set of simple quasi series-parallel matroids of rank $k$ on $E$, and write $\SQSP(n,k):=\SQSP([n],k)$, where $[n]=\{1,\,2,\,\ldots,\,n\}$. Ferroni and Larson's combinatorial interpretation of $P_{B_n}(t)$ is as follows.

\begin{theorem}{\em\cite[Theorem 1.1]{FL-braid}}\label{braid and sqsp}
For any positive integer $n$, the coefficient of $t^i$ in the Kazhdan--Lusztig polynomial $P_{B_n}(t)$ is equal to the cardinality of the set $\SQSP(n-1,n-1-i)$.
\end{theorem}

The matroid $B_n$ has rank $n-1$, therefore the polynomial $P_{B_{2n}}(t)$ has degree at most $n-1$, while the polynomial $P_{B_{2n-1}}(t)$ has degree at most $n-2$.
Using Theorem \ref{braid and sqsp}, Ferroni and Larson gave the following explicit formula for the leading coefficient of $P_{B_{2n}}(t)$, which was originally conjectured in \cite[Appendix]{EPW}.

\begin{theorem}{\em\cite[Corollary 2.12]{FL-braid}}\label{odd top}
For any $n > 1$, the coefficient of $t^{n-1}$ in $P_{B_{2n}}(t)$ is equal to
$$|\SQSP(2n-1,n)| = (2n-1)^{n-2}\cdot (2n-3)!!.$$
\end{theorem}

Our first main result is the following explicit formula for the leading coefficient of $P_{B_{2n-1}}(t)$, which we regard as a complement to Theorem \ref{odd top}.

\begin{theorem}\label{even top}
For any $n > 1$, the coefficient of $t^{n-2}$ in $P_{B_{2n-1}}(t)$ is equal to
$$|\SQSP(2n-2,n)| = (2n-1)^{n-2}\cdot (2n-3)!! -  \frac{(n-2)(n-1)^{n-5}\cdot (2n-1)!}{3\cdot (n-2)!}.$$
\end{theorem}

\begin{remark}
	By looking at the formulas in Theorems \ref{odd top} and \ref{even top}, it is apparent that the key formula is for the difference $|\SQSP(2n-1,n)| - |\SQSP(2n-2,n)|$,
	or equivalently the leading coefficient of the polynomial $P_{B_{2n}}(t) - tP_{B_{2n-1}}(t)$.
	It is not {\em a priori} obvious that this difference should be meaningful or interesting.
\end{remark}

Let $E_n$ denote the number of simple series-parallel matroids of rank $n$ on the set $[2n-2]$.
Ferroni and Larson \cite{FL-braid} noted that the $E_n$ can have large prime factors, and did not give a simple expression for $E_n$. However, they obtained the following expression of  $|\SQSP(2n-2,n)|$ in terms of $E_n$.

\begin{theorem}{\em\cite[Proposition 2.13]{FL-braid}}\label{even top_en}
	For any $n > 1$, the coefficient of $t^{n-2}$ in $P_{B_{2n-1}}(t)$ is equal to
	{$$|\SQSP(2n-2,n)| = E_n+\frac{1}{2}\sum_{j=0}^{n-2}\bin{2n-2}{2j+1}\cdot (2j-1)!!(2j+1)^{j-1}\cdot (2n-2j-5)!!(2n-2j-3)^{n-j-3}.$$}
\end{theorem}

Based on Theorems \ref{even top} and \ref{even top_en}, we obtain
the following explicit formula for $E_n$.

\begin{corollary}\label{cor-ssp}
	For any $n > 1$, we have
	\begin{align*}
		E_n=(2n-1)^{n-2}\cdot (2n-3)!! -  \frac{(n-2)(n-1)^{n-5}\cdot (2n-1)!}{3\cdot (n-2)!}-\frac{(n+1)(n-1)^{n-3}\cdot (2n-3)!}{3\cdot (n-1)!}.
	\end{align*}
\end{corollary}


It is known that the {\bf inverse Kazhdan--Lusztig polynomial} $Q_M(t)\in\N[t]$ of matroid $M$ is a polynomial with non-negative integer coefficients of degree strictly less than half the rank of $M$ \cite{GX,BHMPW2}. 
While we have no concrete combinatorial interpretation of these coefficients along the lines of Theorem \ref{braid and sqsp},we obtain the following explicit formulas for the leading coefficients.

\begin{theorem}\label{Q_leading}
For any $n > 1$,
the coefficient of $t^{n-1}$ in $Q_{\mathrm{B}_{2n}}(t)$ is equal to
  \begin{align*}
  (2n-1)^{n-2}\cdot (2n-3)!!,
  \end{align*}
and the coefficient of $t^{n-2}$ in $Q_{\mathrm{B}_{2n-1}}(t)$ is equal to
  \begin{align*}
  \frac{ (n-1)^{n-5}\cdot (2n-1)!}{3\cdot (n-2)!}.
  \end{align*}
\end{theorem}

\begin{remark}
	Comparing Theorem \ref{Q_leading} with Theorems \ref{odd top} and \ref{even top}, we see that
the leading coefficient of $Q_{\mathrm{B}_{2n}}(t)$ coincides with that of $P_{B_{2n}}(t)$, and
the difference $|\SQSP(2n-1,n)| - |\SQSP(2n-2,n)|$ is the leading coefficient of $(n-2)Q_{\mathrm{B}_{2n-1}}(t)$.
\end{remark}

By Theorem \ref{even top}, Corollary \ref{cor-ssp}, and Theorem \ref{Q_leading}, we see that understanding the difference $|\SQSP(2n-1,n)| - |\SQSP(2n-2,n)|$ is very important for us to determine the leading coefficients of $P_{B_{n}}(t)$ and $Q_{B_{n}}(t)$. We prove Theorem \ref{even top} by constructing a surjective map $$\Phi(n):\SQSP(2n-1,n)\to\SQSP(2n-2,n)$$ and studying the fibers.
Many of the fibers turn out to be singletons, and therefore contribute nothing to the difference $|\SQSP(2n-1,n)| - |\SQSP(2n-2,n)|$.
The remaining fibers can be understood via the enumeration of Husimi graphs.  The construction of the map and analysis of the fibers
take place in Section \ref{fibers}, while the discussion of Husimi graphs takes place in Section \ref{gen}.
The proof of Corollary \ref{cor-ssp} is also presented in Section \ref{gen}, which will involve the binomial identity for the Abel polynomials. Section \ref{sec-inverse} will be devoted to the proof of Theorem \ref{Q_leading}.\\

\vspace{\baselineskip}
\noindent
{\bf Acknowledgments:}
We would like to thank Linyuan Lu and Matthew Xie for very helpful discussions about the leading coefficients of
the Kazhdan-Lusztig polynomials and inverse Kazhdan-Lusztig polynomials of braid matroids of even rank. 

\section{A surjective map from \boldmath{$\SQSP(2n-1,n)$} to \boldmath{$\SQSP(2n-2,n)$}}\label{fibers}

The aim of this section is to express the difference between $|\SQSP(2n-1,n)|$ and $|\SQSP(2n-2,n)|$ in terms of the count of certain combinatorial objects. To this end, we will introduce some definitions and lemmas.

\begin{lemma}\label{the map}
If $n > 1$, $M\in\SQSP(2n-1,n)$, and $M' = M \setminus (2n-1)$, then $M'\in\SQSP(2n-2,n)$.
\end{lemma}

\begin{proof}
The class of simple quasi series-parallel matroids is minor closed by \cite[Proposition 2.6(iii)]{FL-braid},
so we just need to show that the rank cannot decrease.  In other words, we need to show that $2n-1$ is not a coloop of $M$.
This follows from the fact that $M$ is connected, which is a consequence of \cite[Proposition 2.10]{FL-braid}.
\end{proof}

For any $n\geq 2$, Lemma \ref{the map} enables us to define a map
$$\Phi(n):\SQSP(2n-1,n)\to \SQSP(2n-2,n)$$ by sending
$M\in\SQSP(2n-1,n)$ to $M \setminus (2n-1)$.  Our next goal is to show that $\Phi(n)$ is surjective.
To achieve this, it is convenient to use a connection between simple quasi series-parallel matroids and triangular cacti. Recall that a {\bf triangular cactus} is a connected graph with the property that every edge belongs to a unique cycle and all cycles have length 3.
Let $E$ be a finite set of cardinality $2n-1$, and let $\Delta(E)$ denote the set of triangular cacti on $E$.
In \cite[Proposition 2.11]{FL-braid}, the authors constructed a bijection
$\Psi(E)$ from $\SQSP(E,n)$ to $\Delta(E)$.  We will write $\Delta(n) := \Delta([2n-1])$ and $\Psi(n) := \Psi([2n-1])$,
and denote the bijection by
$$\Psi(n):\SQSP(2n-1,n)\to \Delta(n).$$

\begin{lemma}\label{surjective}
For any $n\geq 2$, the deletion map $\Phi(n)$ is surjective.
\end{lemma}

\begin{proof}
We will proceed by induction on $n$.  The cases where $n=2$ or $n=3$ are trivial, thus we may assume that $n>3$.  
Let $M'\in \SQSP(2n-2,n)$ be given.  By \cite[Proposition 2.10]{FL-braid}, either $M' \cong A\oplus B$ for some $A\in\SQSP(2k-1,k)$
and $B\in\SQSP(2(n-k)-1,n-k)$, or $M'$ is connected.

Suppose first that $M'$ is not connected.  That is, we have a subset $S\subset [2n-2]$ of cardinality $2k-1$, a simple quasi series-parallel matroid $A$ of rank $k$ on $S$,
and a simple quasi series-parallel matroid $B$ of rank $n-k$ on $[2n-2]\setminus S$, with the property that $M' = A\oplus B$.
Choose elements $e\in S$ and $f\in [2n-2]\setminus S$, and consider the triangular cactus $G$ on $[2n-1]$ obtained by taking the union of the cactus $\Delta(S)(A)$, the cactus $\Delta([2n-2]\setminus S)(B)$, and the triangle $\{e,f,2n-1\}$.
If we take $M\in\SQSP(2n-1,n)$ to be the unique element with $\Psi(n)(M) = G$, then
$\Phi(n)(M) = M\setminus(2n-1) = M'$.


Now suppose that $M'$ is connected.  We will break the argument into two cases, depending on whether or not
$M'$ is a series extension of a simple series-parallel matroid.  First assume that it is, i.e. that
there exists a cocircuit $\{e,f\}\subset[2n-2]$ of $M'$ such that $M'' := M'/e$ is a simple series-parallel matroid.  Let $E := [2n-2]\setminus\{e\}$.
Note that $M''$ is a simple series-parallel matroid of rank $n-1$ on $E$. Consider the triangular cactus $G'' := \Psi(E)(M'')$.  Let $G$ be the triangular cactus on $[2n-1]$ obtained by taking the union of $G''$ and the triangle $\{e,f,2n-1\}$ and then interchanging $f$ and $2n-1$. If we take $M\in\SQSP(2n-1,n)$ to be the unique matroid with $\Psi(n)(M) = G$, then we have $\Phi(n)(M) = M\setminus(2n-1) = M'$.

Finally, suppose that $M'$ is connected but is not a series extension of a simple series-parallel matroid.
In this case, there must exist a circuit $\{e,f,g\}\subset[2n-2]$ of $M'$ such that $M'':= M'\setminus \{e,f\}$ is a
simple series-parallel matroid of rank $n-1$ on $[2n-2]\setminus\{e,f\}$ and $M'$ can be obtained from $M''$ by first 
adding a new element $e$ that is parallel to $g$ (a parallel extension) and then replacing $e$ with a cocircuit $\{e,f\}$ (a series extension).
We will refer to $M'$ as a 
{\bf triangle extension} of $M''$ at the element $g$.
By our inductive hypothesis, there exists a simple series-parallel matroid $M'''$ of rank $n-1$ on the set $[2n-1]\setminus\{e,f\}$
such that $M''$ is obtained by deleting the element $2n-1$ from $\tilde{M}$.  Let $M$ be the matroid on $[2n-1]$ obtained from $M'''$ as a triangle extension at $g$.
Then $\Phi(n)(M) = M\setminus(2n-1) = M'$.
\end{proof}

We proceed to study the fibers of the surjection $\Phi(n)$. This will be done by a careful analysis of the circuits of elements in $\SQSP(2n-1,n)$. Let $M$ be a matroid on the ground set $E$, and suppose that $C\subset E$ is a circuit.  
We say that an element $e\in E\setminus C$ is a {\bf chord} for $C$ if there exists a subset $S\subset C$ such that $S\cup\{e\}$ and $(C\setminus S)\cup\{e\}$ are both circuits.
If $C$ does not have any chords, we will say that it is {\bf chordless}.  We note that every 3-circuit in a simple matroid is chordless, and a simple matroid is determined by its chordless circuits.
A simple matroid that has no chordless $k$-circuits for any $k\geq 4$ is called {\bf chordal}.
Note that every element of $\SQSP(2n-1,n)$ is chordal.

Let $m$ be a natural number.  Let $\SQSP_m(2n-1,n)\subset \SQSP(2n-1,n)$ be the set of matroids in which the element $2n-1$ is contained
in exactly $m$ 3-circuits, and let $$\SQSP_m(2n-2,n) := \Phi(n)\big(\SQSP_m(2n-1,n)\big).$$
If $M\in\SQSP_m(2n-1,n)$, then $\Phi(n)(M)$ has exactly $\binom{m}{2}$ chordless 4-circuits,
coming from pairs of 3-circuits in $M$ that contain $2n-1$.  This fact, along with Lemma \ref{surjective}, implies that
$\SQSP(2n-2,n)$ is equal to the disjoint union of the sets $\SQSP_m(2n-2,n)$.  We will write
$$\Phi_m(n):\SQSP_m(2n-1,n)\to \SQSP(2n-2,n)$$
to denote the restriction of $\Phi(n)$ to $\SQSP_m(2n-1,n)$.

\begin{lemma}\label{disconnected}
If $M'\in\SQSP_m(2n-2,n)$, then $M'$ is connected if and only if $m>1$.
\end{lemma}

\begin{proof}
Choose $M\in\SQSP_m(2n-1,n)$ such that $M'=\Phi_m(n)(M)$.
We know from \cite[Proposition 2.10]{FL-braid} that $M$ is connected, and we also know that all of the chordless circuits of $M$ are 3-circuits.
If $m>1$, then $2n-1$ is contained in multiple 3-circuits of $M$, so deleting it does not disconnect the matroid.
On the other hand, if $m=1$, then $2n-1$ is contained in a unique 3-circuit $\{e,f,2n-1\}$ of $M$.  This implies that there are no circuits
of $M'$ containing both $e$ and $f$, so $M'$ is disconnected.
\end{proof}

The following lemme characterizes the fibers of $\Phi(n)$ over
the connected elements of $\SQSP(2n-2,n)$.

\begin{lemma}\label{bijection}
Let $n\geq 2$.
The map $\Phi_m(n)$ is a bijection when $m\geq 3$, and it is 3-to-1 when $m=2$.
\end{lemma}

\begin{proof}
Let $M'\in \SQSP_m(2n-2,n)$ be given.  We want to count the  matroids $M\in\SQSP_m(2n-1,n)$ with $\Phi_m(n)(M) = M'$.
Since $M$ must be chordal, it is determined by its 3-circuits.  The 3-circuits of $M$ that do not contain the element $2n-1$
coincide with the 3-circuits of $M'$, hence it is sufficient to think about the 3-circuits of $M$ that contain $2n-1$.

Let $S\subset[2n-2]$ be the union of the $\binom m 2$ chordless 4-circuits of $M'$.  Then the restriction of $M'$ to $S$ is isomorphic to the matroid associated with the complete bipartite graph $K_{2,m}$.  If $m\geq 3$, there is a unique chordal
extension of this matroid, represented by the thagomizer graph $T_m$ \cite{thag}.
This uniquely determines all of the 3-circuits of $M$ that contain $2n-1$.  If $m=2$,
then the matroid associated with $K_{2,2}$ is the uniform matroid of rank three on four elements, and there are three different extensions
of this matroid to a chordal matroid on five elements, corresponding to the three ways to partition the four elements into pairs of subsets of size two.  These determine three different matroids $M\in\SQSP_m(2n-1,n)$
that map to $M'$.
\end{proof}

The above lemma shows that only those fibers of $\Phi(n)$ over elements of $\SQSP_1(2n-2,n)\cup \SQSP_2(2n-2,n)$ can contribute to the difference $|\SQSP(2n-1,n)|-|\SQSP(2n-2,n)|$. As shown below, these contributions can be expressed in terms of the counts of some combinatorial objects constructed from triangular cacti.
Define a {\bf desert} to be a disjoint union of triangular cacti, and a {\bf rooted desert} to be a disjoint union of rooted triangular cacti.
Let $\Des_m(n)$ denote the set of deserts on the vertex set $[2n-2]$ with exactly $2m$ connected components,
and let $\RDes_m(n)$ denote the set of rooted deserts on the vertex set $[2n-2]$ with exactly $2m$ connected components.
We have a map $$\Omega_m(n):\RDes_m(n)\to\Des_m(n)$$ given by forgetting the roots.

Let $\Delta_m(n)\subset\Delta(n)$ denote the set of triangular cacti with the property that the vertex $2n-1$ has degree $2m$, or equivalently
that it is contained in exactly $m$ triangles.  Then the bijection $\Psi(n)$ restricts to a bijection
$$\Psi_m(n):\SQSP_m(2n-1,n)\to\Delta_m(n)$$
for all $m$. We also have a map
$$\Pi_m(n):\Delta_m(n)\to\RDes_m(n)$$
given by deleting the vertex $2n-1$ along with all of the triangles that passed through that vertex, and taking the roots to be the vertices
from the deleted triangles. Note that $\Pi_m(n)$ is surjective for all $m$, and $\Pi_1(n)$ is a bijection.

\begin{lemma}\label{burning bridges}
There is a bijection $\Sigma_2(n):\SQSP_2(2n-2,n)\to\RDes_2(n)$
with the property that the following diagram commutes:
\[
\begin{tikzcd}
    \SQSP_2(2n-1,n) \ar[rr, "\Psi_2(n)", "\cong"'] \ar[dd, two heads,"\Phi_2(n)"'] && \Delta_2(n) \ar[dd, two heads, "\Pi_2(n)"] \\ \\
    \SQSP_2(2n-2,n)\ar[rr, "\Sigma_2(n)", "\cong"'] && \RDes_2(n).
\end{tikzcd}
\]
\end{lemma}

\begin{proof}
Given a matroid $M'\in \SQSP_2(2n-2,n)$, we define $\Sigma_2(n)(M')\in \RDes_2(n)$ to have triangles consisting of the
3-circuits of $M'$ and roots consisting of the unique chordless 4-circuit of $M'$.
\end{proof}

\begin{lemma}\label{m=1}
There is a bijection $\Theta_1(n):\Des_1(n)\to \SQSP_1(2n-2,n)$
with the property that the following diagram commutes:
\[
\begin{tikzcd}
    \SQSP_1(2n-1,n) \ar[rr, "\Psi_1(n)", "\cong"'] \ar[dd, two heads,"\Phi_1(n)"'] && \Delta_1(n) \ar[rr, "\Pi_1(n)", "\cong"'] && \RDes_1(n)\ar[dd, two heads, "\Omega_1(n)"] \\ \\
    \SQSP_1(2n-2,n)\ar[rrrr, "\Theta_1(n)", "\cong"'] &&&& \Des_1(n).
\end{tikzcd}
\]
\end{lemma}

\begin{proof}
Let $M'\in \SQSP_1(2n-2,n)$ be given.
By Lemma \ref{disconnected}, $M'$ is disconnected, so
there exists a subset $S\subset[2n-2]$ of cardinality $2k-1$, a matroid $A\in \SQSP(S,k)$, and another matroid $B\in \SQSP([2n-2]\setminus B, n-k)$ such that $M' = A\oplus B$.  We then define
$\Theta_1(n)(M')$ to be the union of $\Psi(S)(A)$ and $\Psi([2n-2]\setminus S)(B)$.
\end{proof}

We now come to the main result of this section.

\begin{proposition}\label{difference}
For any $n>1$, we have
$$|\SQSP(2n-1,n)| - |\SQSP(2n-2,n)| = 2 \cdot |\RDes_2(n)| + |\RDes_1(n)| - |\Des_1(n)|.$$
\end{proposition}

\begin{proof}
We have
$$|\SQSP(2n-1,n)| - |\SQSP(2n-2,n)| = \sum_{m\geq 1} \Big(|\SQSP_m(2n-1,n)| - |\SQSP_m(2n-2,n)|\Big).$$
When $m\geq 3$, Lemma \ref{bijection} tells us that $\Phi_m(n):\SQSP_m(2n-1,n)\to \SQSP_m(2n-2,n)$ is a bijection,
thus the summand indexed by $m$ vanishes.  When $m=2$,
Lemma \ref{bijection} tells us that the map $\Phi_2(n):\SQSP_2(2n-1,n)\to \SQSP_2(2n-2,n)$
is 3-to-1, and Lemma \ref{burning bridges} identifies $\SQSP_2(2n-2,n)$ with $\RDes_2(n)$.
This implies that
$$|\SQSP_2(2n-1,n)| - |\SQSP_2(2n-2,n)| = 2|\SQSP_2(2n-2,n)| = 2 |\RDes_2(n)|.$$
Finally, when $m=1$,
Lemma \ref{m=1} identifies the map $\Phi_1(n):\SQSP_1(2n-1,n)\to\SQSP_1(2n-2,n)$ with the map $\Omega_1(n):\RDes_1(n)\to \Des_1(n)$.
The result follows.
\end{proof}

Thus, to give an explicit formula for computing $|\SQSP(2n-1,n)| - |\SQSP(2n-2,n)|$, it remains to determine $|\RDes_2(n)|, |\RDes_1(n)|$ and  $|\Des_1(n)|$. This task will be completed in the next section, via the enumeration of Husimi graphs.

\section{Proofs of Theorem \ref{even top} and Corollary \ref{cor-ssp}}\label{gen}
A {\bf block} of a graph is a maximal 2-connected subgraph.  A {\bf Husimi graph} is a connected graph whose blocks
are all isomorphic to complete graphs. 
We say that it is of type $(n_2,n_3,n_4,\ldots)$, where $n_i$ is the number of blocks isomorphic to $K_i$.
For any $p\geq 1$, let $\tau_p(n_2,n_3,n_4,\ldots)$ denote the number of Husimi graphs of  type $(n_2,n_3,n_4,\ldots)$
on the vertex set $[p]$.
The following result was initially discovered by Husimi \cite{H-50} and later given a rigorous mathematical proof by Leroux \cite{Leroux}.
See \cite[Lemma 5.3.3]{Okoth} for a clear statement and discussion of the history.

\begin{lemma}\label{thm-num-Husimi}
    For any $p\geq 1$, we have
     \begin{align*}
        \tau_p(n_2,n_3,n_4,\ldots)=\frac{p!}{\prod_{i=2}^p[(i-1)!]^{n_i}\, n_i!}\,p^{-2+\sum_{i=2}^pn_i}.
    \end{align*}
\end{lemma}

Note that a Husimi graph of type $(p-1,0,0,\ldots)$ is just a tree on the vertex set $[p]$, and Lemma \ref{thm-num-Husimi}
specializes to the statement, due originally to Cayley, that the number of such trees is $p^{\,p-2}$.
Similarly, a Husimi graph of type $(0,r-1,0,0,\ldots)$ is
a triangular cactus on the set $[2r-1]$.  In this case, 
Lemma \ref{thm-num-Husimi} says that
\begin{align}\label{eq-cactus}
|\Delta(r)|= \frac{(2r-1)^{r-3}\cdot (2r-1)!}{2^{r-1}\cdot (r-1)!}.
\end{align}


\begin{proposition}\label{prop_rooted-deserts}
For any $n>1$ and $m\geq 1$, we have
$$|\RDes_m(n)|=\frac{(n-1)^{n-m-2}\cdot (2n-2)!}{2\cdot (2m-1)!\cdot (n-m-1)!}.$$
In particular,
$$|\RDes_1(n)|=\frac{(n-1)^{n-3}\cdot (2n-2)!}{2\cdot (n-2)!} \qquad \mbox{and } \qquad |\RDes_2(n)|=\frac{(n-1)^{n-4}\cdot (2n-2)!}{12\cdot (n-3)!}.$$
\end{proposition}

\begin{proof}
Let $\operatorname{HT}_m(n)$ denote the set of Husimi graphs on the vertex set $[2n-2]$ with $n-m-1$ triangular blocks and one block isomorphic to $K_{2m}$.
There is a bijection from $\operatorname{HT}_m(n)$ to $\RDes_m(n)$ that takes a Husimi graph to the rooted desert obtained
by deleting the edges of $K_{2m}$ and taking its vertices as the roots. The result then follows from Lemma \ref{thm-num-Husimi}.
\end{proof}

We proceed to determine $|\Des_1(n)|$. Before that, let us recall a result on the {\bf Abel polynomials}
$$A_m(x;a):=x(x-am)^{m-1}.$$
Abel \cite[Section 2.6]{Roman-1984} showed that these polynomials
satisfy the identity
\begin{align*}
\sum_{j=0}^m\binom{m}{j}A_j(x;a)A_{m-j}(y;a)=A_m(x+y;a)
\end{align*}
for any integers $m,x,y$ and $a$.
By combining this formula with the definition of the Abel polynomials, we obtain the equation
\begin{align}\label{eq-abel}
\sum_{j=0}^m\binom{m}{j}xy(x-aj)^{j-1}(y-am+aj)^{m-j-1}=(x+y)(x+y-am)^{m-1}.
\end{align}
Differentiating both sides of \eqref{eq-abel} with respect to $x$, we get
\begin{align}\label{eq-abel-x}
\sum_{j=0}^m\binom{m}{j}j(x-a)y(x-aj)^{j-2}(y-am+aj)^{m-j-1}
=m(x+y-a)(x+y-am)^{m-2}.
\end{align}
With these formulas, we are able to give the following explicit formula for $|\Des_1(n)|$.

\begin{lemma}\label{lem-abel-2}
For $n>1$, we have
\begin{align*}
  |\Des_1(n)|=\frac{(n+1)(n-1)^{n-5}\cdot (2n-2)!}{6\cdot (n-2)!}.
\end{align*}
\end{lemma}
\begin{proof}
  For any $m\geq 1$, an element of $\Des_m(n)$ consists of a partition of $[2n-2]$ into $2m$ parts and a triangular cactus on each of those parts.
  This implies that
  $$(2m)!\, |\Des_m(n)| \;\;= \sum_{(2k_1-1)+\cdots+(2k_{2m}-1) = 2n-2}\binom{2n-2}{2k_1-1,\ldots,2k_{2m}-1} \prod_{i=1}^{2m} |\Delta(k_i)|,$$
  where the factor of $(2m)!$ reflects the fact that the parts of the partition are unordered.
  When $m=1$, the above formula simplifies to
  $$|\Des_1(n)| = \frac{1}{2}\sum_{r= 1}^{n-1}\binom{2n-2}{2r-1}\, |\Delta(r)|\cdot |\Delta(n-r)|.$$
Substituting \eqref{eq-cactus} into the right hand side and reindexing, we obtain the formula
  \begin{align}%
    |\Des_1(n)|&=\frac{1}{2}\sum_{r=0}^{n-2}\bin{2n-2}{2r+1}\cdot (2r-1)!!(2r+1)^{r-1}\cdot (2n-2r-5)!!(2n-2r-3)^{n-r-3}\label{equ-cx2}\\[5pt]
    &=\sum_{r=0}^{n-2}\frac{(2n-2)!(2r+1)^{r-2}(2n-2r-3)^{n-r-4}}{2^{n-1}\cdot r!\cdot (n-r-2)!}.\nonumber
    \end{align}
  Now it suffices to show that
    \begin{align*}
      \sum_{r=0}^{n-2}\frac{(2n-2)!(2r+1)^{r-2}(2n-2r-3)^{n-r-4}}{2^{n-1}\cdot r!\cdot (n-r-2)!}=\frac{(n+1)(n-1)^{n-5}\cdot (2n-2)!}{6\cdot (n-2)!},
    \end{align*}
    or equivalently that
    \begin{align}\label{eq-aim-identity}
      \sum_{r=0}^{n-2}3\binom{n-2}{r}(2r+1)^{r-2}(2n-2r-3)^{n-r-4} = 2^{n-2}(n+1)(n-1)^{n-5}.
    \end{align}
    To this end, we take $m=n-2$, $a=-2$, $x=1$, and $y=1$ in Equations \eqref{eq-abel} and \eqref{eq-abel-x} to obtain the following
    two equations:
    \begin{align}
      &\sum_{r=0}^{n-2}\binom{n-2}{r}(2r+1)^{r-1}(2n-2r-3)^{n-r-3}=2^{n-2}(n-1)^{n-3},\label{eq-abel-s-f}\\[6pt]
      &\sum_{r=0}^{n-2}\binom{n-2}{r}3r(2r+1)^{r-2}(2n-2r-3)^{n-r-3}=2^{n-2}(n-2)(n-1)^{n-4}\label{eq-abel-x-s-f}.
      \end{align}
    Substituting $r$ for $n-2-r$ into the left hand side of \eqref{eq-abel-x-s-f} yields
    \begin{align}
      \sum_{r=0}^{n-2}\binom{n-2}{r}3(n-r-2)(2r+1)^{r-1}(2n-2r-3)^{n-r-4}=2^{n-2}(n-2)(n-1)^{n-4}.\label{eq-abel-y-s-f}
    \end{align}
    By subtracting \eqref{eq-abel-x-s-f} and \eqref{eq-abel-y-s-f} from \eqref{eq-abel-s-f} multiplied by $3$, we obtain
    the desired \eqref{eq-aim-identity}. This completes the proof.
    \end{proof}

We are now ready to prove Theorem \ref{even top}.

\begin{proof}[Proof of Theorem \ref{even top}.]
  Let $g_n := |\SQSP(2n-1,n)| - |\SQSP(2n-2,n)|$.
  By Proposition \ref{difference}, Proposition \ref{prop_rooted-deserts}, and Lemma \ref{lem-abel-2}, we have
  \begin{align*}
    g_n &=2 \cdot |\RDes_2(n)| + |\RDes_1(n)| - |\Des_1(n)|\\
    &=\frac{(n-1)^{n-4}\cdot (2n-2)!}{6\cdot (n-3)!} + \frac{(n-1)^{n-3}\cdot (2n-2)!}{2\cdot (n-2)!} - \frac{(n+1)(n-1)^{n-5}\cdot (2n-2)!}{6\cdot (n-2)!}\\
    &= (2n-2)!\cdot \frac{ (n-2)(n-1)^{n-4} + 3(n-1)^{n-3} - (n+1)(n-1)^{n-5}}{6\cdot (n-2)!}\\
    &= \frac{2(n-2)(n-1)^{n-5}\cdot (2n-1)!}{6\cdot (n-2)!}\\
    &= \frac{(n-1)^{n-5}\cdot (2n-1)!}{3\cdot (n-3)!}.
  \end{align*}
  Combining this with Theorem \ref{odd top} gives the result.
\end{proof}

Finally, we prove Corollary \ref{cor-ssp}.

\begin{proof}[Proof of Corollary \ref{cor-ssp}.]
By Equation \eqref{equ-cx2} and Theorem \ref{even top_en}, we find that
 \begin{align*}
    E_n=|\SQSP(2n-2,n)|-|\Des_1(n)|.
  \end{align*}
Then combining Lemma \ref{lem-abel-2} and  Theorem \ref{even top} gives the desired result.
\end{proof}

\section{Proof of Theorem \ref{Q_leading}}\label{sec-inverse}

The aim of this section is to prove Theorem \ref{Q_leading}.
To this end, we need to use a relation between $Q_{B_n}(t)$ and $P_{B_n}(t)$.
Before recalling this relation, we will follow \cite{Gordon2012MatroidsAG} to introduce some notation from matroid theory.

Let $M=(E,\mathcal{F})$ be a loopless matroid on ground set $E$ with the set of flats $\mathcal{F}$. The lattice of flats of $M$ is denoted by $\mathscr{L}(M)$. For any flat $F$ of $M$, let $M|_F$ denote the restriction of $M$ to $F$, and 
let $M/F$ denote the matroid obtained from $M$ by contracting $F$.
For any subset $I$ of $E$, let $\mathrm{rk}\, I$ denote the rank of $I$ in the matroid $M$. The rank of matroid $M$, denoted by $\mathrm{rk}\, M$, is defined to be $\mathrm{rk}\, E$.
Gao and Xie \cite[Theorem 1.3]{GX} established the following relation between $Q_{M}(t)$ and $P_{M}(t)$:
\begin{align}\label{from_Q_to_P}
P_M(t)=-\sum_{F\in \mathscr{L}(M)\backslash \{E\}}P_{M|_F}(t) \cdot (-1) ^{\mathrm{rk} \, M/ F} Q_{M/ F}(t).
\end{align}

Let $[t^i]f(t)$ denote the coefficient $t^i$ in the polynomial $f(t)$.
Based on \eqref{from_Q_to_P}, Vecchi \cite[Theorem 4.1]{vecchi2021matroid} showed that, for any matroid $M$ of odd rank $2m-1$, we have the identity
\begin{align*}
  {[t^{m-1}]}P_M(t)={[t^{m-1}]}Q_M(t).
\end{align*}
Since the rank of braid matroid $B_{2n}$ is $2n-1$, we have
\begin{align}\label{eq-braid-pq}
[t^{n-1}]P_{\mathrm{B}_{2n}}(t)=[t^{n-1}]Q_{\mathrm{B}_{2n}}(t).
\end{align}
The relationship between the leading coefficients of $P_{B_{2n-1}}(t)$ and $Q_{B_{2n-1}}(t)$ is more subtle; the precise formula appears in the following lemma.

\begin{lemma}\label{odd_P_Q_relation}
For any $n> 1$, we have
\begin{align}\label{relation-P-and-Q}
[t^{n-2}]P_{\mathrm{B}_{2n-1}}(t)+[t^{n-2}]Q_{\mathrm{B}_{2n-1}}(t)
=\sum_{j=1}^{n-1} \binom{2n-1}{2j}
[t^{j-1}]P_{\mathrm{B}_{2j}}(t)   \cdot [t^{n-1-j}]Q_{\mathrm{B}_{2n-2j}}(t).
\end{align}
\end{lemma}

\begin{proof}
Taking $M$ to be $B_{2n-1}$ in Equation \eqref{from_Q_to_P} yields
\begin{align}\label{P_Q_relation_coeff}
P_{\mathrm{B}_{2n-1}}(t)+Q_{\mathrm{B}_{2n-1}}(t)
=-\sum_{F\in \mathscr{L}(\mathrm{B}_{2n-1})\backslash \{\emptyset,E\}}
P_{\mathrm{B}_{2n-1}|_F}(t) \cdot (-1) ^{\mathrm{rk} \, \mathrm{B}_{2n-1}/F} Q_{\mathrm{B}_{2n-1}/F}(t) .
\end{align}
We now compare coefficients of $t^{n-2}$ on both sides of Equation \eqref{P_Q_relation_coeff}. It suffices to show that
\begin{align}
\sum_{F\in \mathscr{L}(\mathrm{B}_{2n-1})\backslash \{\emptyset,E\}}
[t^{n-2}]&\Big(P_{\mathrm{B}_{2n-1}|_F}(t) \cdot (-1) ^{\mathrm{rk} \, \mathrm{B}_{2n-1}/F} Q_{\mathrm{B}_{2n-1}/F}(t) \Big)\nonumber\\
=&-\sum_{j=1}^{n-1} \binom{2n-1}{2j}
[t^{j-1}]P_{\mathrm{B}_{2j}}(t)\cdot [t^{n-1-j}]Q_{\mathrm{B}_{2n-2j}}(t).\label{eq-equatingcoeff}
\end{align}

The lattice of $\mathscr{L}(B_k)$ is isomorphic to the lattice of set-theoretic partitions of the set $[k]$, with the minimal element $\emptyset$ corresponding to the partition of $[k]$ into $k$ singletons
and the maximal element $E$ corresponding to the partition of $[k]$ into a single part.
We say that $F \in \mathscr{L}(\mathrm{B}_k)$ is of type $\lambda$ if
the partition $\lambda$ can be obtained by arranging the sizes of the blocks of the corresponding set partition
in descending order. If $F$ is of type $\lambda=(\lambda_1,\lambda_2,\ldots,\lambda_{\ell(\lambda)})\vdash k$, then (after simplification) we have
\begin{align}\label{eq-isomorphism}
\mathrm{B}_{k}|_F \cong \mathrm{B}_{\lambda_1}\oplus\mathrm{B}_{\lambda_2}\oplus \dots\oplus\mathrm{B}_{\lambda_{\ell(\lambda)}} \quad \mbox{ and } \quad \mathrm{B}_{k}/F\cong \mathrm{B}_{\ell(\lambda)}.
\end{align}
By {\cite[Theorem 2.2 and Proposition 2.7]{EPW}} and \cite[Theorem 1.2]{GX}, we have
\begin{align}\label{eq-isomorphism-kl}
P_{\mathrm{B}_{k}|_F}(t)=P_{\mathrm{B}_{\lambda_1}}(t) P_{\mathrm{B}_{\lambda_2}}(t) \cdots P_{\mathrm{B}_{\lambda_{\ell(\lambda)}}}(t) \quad \mbox{ and } \quad Q_{\mathrm{B}_{k}/F}(t)=Q_{\mathrm{B}_{\ell(\lambda)}}(t).
\end{align}

Let $F$ be a nonempty proper flat of $B_{2n-1}$, and let $\lambda\vdash 2n-1$ be the type of $F$.  We will analyze the summand of Equation \eqref{eq-equatingcoeff} indexed by $F$ according to the following cases. 
\\\\
Case I: $\lambda=(2j,1^{2n-1-2j})$ for some $1\leq j \leq n-1$. By Equation \eqref{eq-isomorphism-kl} and the fact $P_{B_1}(t)=1$, we have
\begin{align*}
P_{\mathrm{B}_{2n-1}|_F}(t) \cdot (-1) ^{\mathrm{rk} \, \mathrm{B}_{2n-1}/F} Q_{\mathrm{B}_{2n-1}/F}(t)=
-P_{\mathrm{B}_{2j}}(t)\cdot Q_{\mathrm{B}_{2n-2j}}(t).
\end{align*}
Since $\deg P_{\mathrm{B}_{2j}}(t)  \leq j-1$ and $\deg Q_{\mathrm{B}_{2n-2j}}(t) \leq n-j-1$, we have
\begin{align*}
[t^{n-2}] \Big(P_{\mathrm{B}_{2n-1}|_F}(t) \cdot (-1) ^{\mathrm{rk} \, \mathrm{B}_{2n-1}/F} Q_{\mathrm{B}_{2n-1}/F}(t)\Big)
=-[t^{j-1}]P_{\mathrm{B}_{2j}}(t)   \cdot [t^{n-j-1}]Q_{\mathrm{B}_{2n-2j}}(t).
\end{align*}
Case II: $\lambda=(2j-1,1^{2n-2j})$ for some $2\leq j \leq n-1$.  This time, we have
\begin{align*}
P_{\mathrm{B}_{2n-1}|_F}(t) \cdot (-1) ^{\mathrm{rk} \, \mathrm{B}_{2n-1}/F} Q_{\mathrm{B}_{2n-1}/F}(t)=
P_{\mathrm{B}_{2j-1}}(t)\cdot Q_{\mathrm{B}_{2n-2j+1}}(t).
\end{align*}
Since $\deg P_{\mathrm{B}_{2j}}(t)  \leq j-2$ and $\deg Q_{\mathrm{B}_{2n-2j+1}}(t) \leq n-j-1$, we have
\begin{align*}
[t^{n-2}] \Big(P_{\mathrm{B}_{2n-1}|_F}(t) \cdot (-1) ^{\mathrm{rk} \, \mathrm{B}_{2n-1}/F} Q_{\mathrm{B}_{2n-1}/F}(t)\Big)
=0.
\end{align*}
Case III: $\lambda=(\lambda_1,\lambda_2,\ldots,\lambda_{i},1^{2n-1-\sum_{j=1}^i\lambda_j})$ for some $i\geq 2$ and $\lambda_i>1$. Now we have
\begin{align*}
P_{\mathrm{B}_{2n-1}|_F}(t) \cdot (-1) ^{\mathrm{rk} \, \mathrm{B}_{2n-1}/F} Q_{\mathrm{B}_{2n-1}/F}(t)=
(-1)^{2n+i-\sum_{j=1}^i\lambda_j}P_{\mathrm{B}_{\lambda_1}}(t)\cdots P_{\mathrm{B}_{\lambda_i}}(t)\cdot Q_{\mathrm{B}_{2n+i-1-\sum_{j=1}^i\lambda_j}}(t).
\end{align*}
Since $\deg P_{\mathrm{B}_{k}}(t)\leq \frac{k-2}{2}$ and $\deg Q_{\mathrm{B}_{k}}(t)\leq \frac{k-2}{2}$ for any $k\geq 2$,
$$\deg (P_{\mathrm{B}_{\lambda_1}}(t)\cdots P_{\mathrm{B}_{\lambda_i}}(t))= \deg P_{\mathrm{B}_{\lambda_1}}(t)+\cdots+\deg P_{\mathrm{B}_{\lambda_i}}(t) \leq \frac{\sum_{j=1}^i\lambda_j-2i}{2}$$
and
$$\deg Q_{\mathrm{B}_{2n+i-1-\sum_{j=1}^i\lambda_j}}(t) \leq \frac{2n+i-3-\sum_{j=1}^i\lambda_j}{2}.$$
Since $i\geq 2$, we have
\begin{align*}
  \deg P_{\mathrm{B}_{\lambda_1}}(t)\cdots P_{\mathrm{B}_{\lambda_i}}(t)+\deg Q_{\mathrm{B}_{2n+i-1-\sum_{j=1}^i\lambda_j}}(t)\leq \frac{2n-i-3}{2}\leq \frac{2n-5}{2}<n-2.
\end{align*}
Thus
\begin{align*}
[t^{n-2}] \Big(P_{\mathrm{B}_{2n-1}|_F}(t) \cdot (-1) ^{\mathrm{rk} \, \mathrm{B}_{2n-1}/F} Q_{\mathrm{B}_{2n-1}/F}(t)\Big)
=0.
\end{align*}

Combining the above three cases, we find that only those flats of type $\lambda=(2j,1^{2n-1-2j})$ can contribute
to the left hand side of \eqref{eq-equatingcoeff}. Note that, for each $1\leq j \leq n-1$, there are exactly $\binom{2n-1}{2j}$ flats of type $\lambda=(2j,1^{2n-1-2j})$. This completes the proof of Equation \eqref{eq-equatingcoeff}, and hence that of the lemma.
\end{proof}

Now we are ready to prove Theorem \ref{Q_leading}.

\begin{proof}[Proof of Theorem \ref{Q_leading}.]
By Equation \eqref{eq-braid-pq} and Theorem \ref{odd top}, we see that
\begin{align}\label{eq-even-Q-leading}
  [t^{n-1}]P_{\mathrm{B}_{2n}}(t)=[t^{n-1}]Q_{\mathrm{B}_{2n}}(t)
=(2n-1)^{n-2}\cdot (2n-3)!!=\frac{(2n-1)!(2n-1)^{n-3}}{2^{n-1}\cdot(n-1)!},
\end{align}
and by Theorem \ref{even top}, we have
\begin{align}\label{eq-odd-P-leading}
  [t^{n-2}]P_{\mathrm{B}_{2n-1}}(t)=\frac{(2n-1)!(2n-1)^{n-3}}{2^{n-1}\cdot (n-1)!} -  \frac{(n-2)(n-1)^{n-5}\cdot (2n-1)!}{3\cdot (n-2)!}.
\end{align}
It remains only to show that
\begin{align*}
[t^{n-2}]Q_{\mathrm{B}_{2n-1}}(t)
=\frac{(2n-1)! (n-1)^{n-5}}{3\cdot (n-2)!}.
\end{align*}
By Equations \eqref{eq-even-Q-leading}, \eqref{eq-odd-P-leading}, and \eqref{relation-P-and-Q}, this is equivalent to the statement that
\begin{align*}
\frac{(2n-1)! (n-1)^{n-5}}{3\cdot (n-2)!}
=&\sum_{j=1}^{n-1} \binom{2n-1}{2j}
\frac{(2j-1)!(2j-1)^{j-3}}{2^{j-1}\cdot (j-1)!}
\cdot\frac{(2n-2j-1)!(2n-2j-1)^{n-j-3}}{2^{n-j-1}\cdot (n-j-1)!}\\
&-\Big( \frac{(2n-1)!(2n-1)^{n-3}}{2^{n-1}\cdot (n-1)!}-\frac{(2n-1)!(n-1)^{n-5}}{3\cdot (n-3)!} \Big),
\end{align*}
which simplifies further to the equation
\begin{align}\label{eq-3}
\sum_{j=0}^{n-1} 3\binom{{n-1}}{j}{(2j-1)^{j-3}}{(2n-2j-1)^{n-j-3}}
=-8(n-3)(2n-2)^{n-4}.
\end{align}
To prove Equation \eqref{eq-3}, we will first establish the following two identities:
\begin{align}
  \sum_{j=0}^m 3\binom{m}{j}(2j-1)^{j-2}(2m-2j+1)^{m-j-2}&=8(2m)^{m-2}, \label{comb-12-1}\\
  \sum_{j=0}^m 3\binom{m}{j}(4mj-4j+1)(2j-1)^{j-3}(2m-2j+1)^{m-j-2}
  &=8(2m^2+m-2)(2m)^{m-3}.\label{comb-23}
\end{align}

Let us first prove Equation \eqref{comb-12-1}.
Differentiating both sides of Equation \eqref{eq-abel} with respect to $y$, we obtain the identity
\begin{align}  \label{eq-abel-y}
\sum_{j=0}^m\binom{m}{j}(m-j)x(y-a)(x-aj)^{j-1}(y-am+aj)^{m-j-2}
=m(x+y-a)(x+y-am)^{m-2}.
\end{align}
Letting $a=-2, x=-1, y=1$ in Equations \eqref{eq-abel}, \eqref{eq-abel-x}, and \eqref{eq-abel-y}, we obtain the following:
\begin{align}
\sum_{j=0}^m\binom{m}{j}(2j-1)^{j-1}(2m-2j+1)^{m-j-1}
&=0,\label{eq-abel-s}\\[6pt]
\sum_{j=0}^m \binom{m}{j}j(2j-1)^{j-2}(2m-2j+1)^{m-j-1}
&=(2m)^{m-1},\label{eq-abel-x-s}\\[6pt]
\sum_{j=0}^m -3\binom{m}{j}(m-j)(2j-1)^{j-1}(2m-2j+1)^{m-j-2}
&=(2m)^{m-1}.\label{eq-abel-y-s}
\end{align}

Now, subtracting Equation \eqref{eq-abel-s} from Equation \eqref{eq-abel-x-s} multiplied by $2$ yields
\begin{align}\label{comb-1}
\sum_{j=0}^m\binom{m}{j} (2j-1)^{j-2}(2m-2j+1)^{m-j-1}&=2(2m)^{m-1},
\end{align}
and adding Equation \eqref{eq-abel-s} multiplied by $3$ to Equation \eqref{eq-abel-y-s} multiplied by $2$ yields
\begin{align}\label{comb-2}
\sum_{j=0}^m&3\binom{m}{j} (2j-1)^{j-1}(2m-2j+1)^{m-j-2}=2(2m)^{m-1}.
\end{align}
Furthermore, adding Equation  \eqref{comb-1}  multiplied by $3$ to Equation \eqref{comb-2} and then cancelling the common factor $2m$ lead to
the desired Equation \eqref{comb-12-1}.

In the same manner we can prove Equation \eqref{comb-23}. Differentiating Equation \eqref{eq-abel-y} with respect to $x$, we have
\begin{align}
\sum_{j=0}^m\binom{m}{j}j(m-j)(x-a)(y-a)&(x-aj)^{j-2}(y-am+aj)^{m-j-2}\nonumber\\
&
=m(m-1)(x+y-2a)(x+y-am)^{m-3}.\label{eq-abel-y-x}
\end{align}
Differentiating again gives
\begin{align}
\sum_{j=0}^m\binom{m}{j}j(j-1)(m-j)(x-2a)(y-a)
&(x-aj)^{j-3}(y-am+aj)^{m-j-2}\nonumber\\
&
=m(m-1)(m-2)(x+y-3a)(x+y-am)^{m-4}.\label{eq-abel-y-x-x}
\end{align}
Letting $a=-2, x=-1, y=1$ in Equations \eqref{eq-abel-y-x} and \eqref{eq-abel-y-x-x}, we obtain
\begin{align}
\sum_{j=0}^m3\binom{m}{j}j(m-j)(2j-1)^{j-2}(2m-2j+1)^{m-j-2}
&=2(m-1)(2m)^{m-2},\label{eq-abel-y-x-s}\\[6pt]
\sum_{j=0}^m 3\binom{m}{j}j(j-1)(m-j)(2j-1)^{j-3}(2m-2j+1)^{m-j-2}
&=(m-1)(m-2)(2m)^{m-3}.\label{eq-abel-y-x-x-s}
\end{align}
Then, subtracting Equation \eqref{eq-abel-y-x-x-s} multiplied by $2$ from Equation \eqref{eq-abel-y-x-s} yields
\begin{align}\label{comb-3}
  \sum_{j=0}^m3\binom{m}{j}j(m-j)(2j-1)^{j-3}(2m-2j+1)^{m-j-2} =2(m-1)(m+2)(2m)^{m-3}.
\end{align}
Adding Equation \eqref{comb-2} to Equation \eqref{comb-3} multiplied by $4$, we get Equation \eqref{comb-23}, as desired.

Now we can derive Equation \eqref{eq-3} from Equations \eqref{comb-12-1} and \eqref{comb-23}. By subtracting Equation \eqref{comb-12-1} multiplied by $2m-2$ from Equation \eqref{comb-23} and then cancelling the common factor $2m-1$, we find that
\begin{align*}
&\sum_{j=0}^m3\binom{m}{j} {(2j-1)^{j-3}}{\big(2m-2j+1\big)^{m-j-2}}
=-8 (m-2)(2m)^{m-3}.
\end{align*}
Substituting $m$ to $n-1$ in the above formula yields Equation \eqref{eq-3}.
This completes the proof.
\end{proof}

\bibliography{./symplectic}

\newcommand{\etalchar}[1]{$^{#1}$}
\def\cprime{$'$}
\providecommand{\bysame}{\leavevmode\hbox to3em{\hrulefill}\thinspace}
\providecommand{\MR}{\relax\ifhmode\unskip\space\fi MR }
\providecommand{\MRhref}[2]{%
  \href{http://www.ams.org/mathscinet-getitem?mr=#1}{#2}
}
\providecommand{\href}[2]{#2}
\begin{thebibliography}{BHM{\etalchar{+}}20}

\bibitem[BHM{\etalchar{+}}20]{BHMPW2}
Tom Braden, June Huh, Jacob Matherne, Nicholas Proudfoot, and Botong Wang,
  \emph{{Singular Hodge theory for combinatorial geometries}}, 2020,
  \textsf{arXiv:2010.06088}.

\bibitem[EPW16]{EPW}
Ben Elias, Nicholas Proudfoot, and Max Wakefield, \emph{The {K}azhdan-{L}usztig
  polynomial of a matroid}, Adv. Math. \textbf{299} (2016), 36--70.

\bibitem[FL23]{FL-braid}
Luis Ferroni and Matt Larson, \emph{Kazhdan-{L}usztig polynomials of braid
  matroids}, 2023, \textsf{arXiv:2303.02253}.

\bibitem[Ged17]{thag}
Katie Gedeon, \emph{Kazhdan-{L}usztig polynomials of thagomizer matroids},
  Electron. J. Combin. \textbf{24} (2017), no.~3.

\bibitem[GM12]{Gordon2012MatroidsAG}
Gary Gordon and Jennifer McNulty, \emph{Matroids: {A} {G}eometric
  {I}ntroduction}, Combridge University Press, 2012.

\bibitem[GX21]{GX}
Alice L.~L Gao and Matthew H.~Y. Xie, \emph{The inverse {K}azhdan-{L}usztig
  polynomial of a matroid}, J. Combin. Theory Ser. B. \textbf{151} (2021),
  375--392.

\bibitem[Hus50]{H-50}
Kodi Husimi, \emph{Note on {M}ayers' theory of cluster integrals}, J. Chem.
  Phys. \textbf{18} (1950), no.~5, 682--684.

\bibitem[Ler04]{Leroux}
Pierre Leroux, \emph{Enumerative problems inspired by {M}ayer's theory of
  cluster integrals}, Electron. J. Combin. \textbf{11} (2004), no.~1, Research
  Paper 32, 28.

\bibitem[Oko15]{Okoth}
Isaac~Owino Okoth, \emph{Combinatorics of oriented trees and tree-like
  structures}, 2015.

\bibitem[Rom05]{Roman-1984}
Steven Roman, \emph{The {U}mbral {C}alculus}, Dover Publications, INC.
  Mineola-New York, 2005.

\bibitem[Vec21]{vecchi2021matroid}
Lorenzo Vecchi, \emph{On matroid modularity and the coefficients of the inverse
  {K}azhdan-{L}usztig polynomial of a matroid}, 2021.

\end{thebibliography}
\bibliographystyle{amsalpha}

\end{document}